\documentclass[aap]{imsart}
\RequirePackage[OT1]{fontenc}
\RequirePackage{amsthm,amsmath,amssymb}
\RequirePackage[numbers]{natbib}
\RequirePackage[colorlinks,citecolor=blue,urlcolor=blue]{hyperref}
\arxiv{math.PR/0000000}

\usepackage{tikz}
\usetikzlibrary{arrows}

\arxiv{arXiv:0000.0000}

\startlocaldefs
\numberwithin{equation}{section}
\theoremstyle{plain}
\newtheorem{thm}{Theorem}[section]
\newtheorem{definition}[thm]{Definition}
\newtheorem{theorem}{Theorem}

\newtheorem{corollary}[theorem]{Corollary}

\newtheorem{lemma}[theorem]{Lemma}

\endlocaldefs

\begin{document}

\begin{frontmatter}

\title{Information weighted sampling for detecting rare items in finite populations with a focus on security}
\runtitle{Information Weighted Sampling}

\begin{aug}
\author{\fnms{Andr\'e J.} \snm{Hoogstrate}\corref{}\thanksref{u1,T1}\ead[label=e1]{a.hoogstrate@nfi.minvenj.nl}\ead[label=e3]{a.j.hoogstrate@cdh.leidenuniv.nl}}
\and
\author{\fnms{Chris A.J.} \snm{Klaassen}\thanksref{u2}\ead[label=e2]{c.a.j.klaassen@uva.nl}}

\thankstext{T1}{Part of the research was funded by the program VIA: Veiligheidsverbetering door Information
 Awareness (Improving Security by Information Awareness).}
\runauthor{A.J. Hoogstrate et al.}

\affiliation{Netherlands Forensic Institute and Leiden University\thanksmark{u1} \\ University of Amsterdam\thanksmark{u2}}

\address{Knowledge and Expertise Centre \\\qquad for Intelligent Data Analysis,\\
            Netherlands Forensic Institute,\\
            Laan van Ypenburg 6,\\
            2497 GB The Hague, The Netherlands,\\
            \printead{e1}}
\address{Centre for Terrorism and Counterterrorism,\\
         Campus The Hague, Leiden University,\\
         Kantoren Stichthage,\\
         Koningin Julianaplein 10,\\
         P.O.Box 13228,\\
         2501 EE The Hague, The Netherlands,\\
     \printead{e3}}
\address{Korteweg-de Vries Institute for Mathematics\\
 University of Amsterdam,\\
 Science Park 904, P.O. Box 94248,\\
 1090 GE Amsterdam, The Netherlands,\\
          \printead{e2}}

\end{aug}

\begin{abstract}
Frequently one has to search within a finite
population for a single particular individual or item with a rare
characteristic. Whether an item possesses the characteristic can
only be determined by close inspection. The availability of
additional information about the items in the population opens the
way to a more effective search strategy than just random sampling
or complete inspection of the population. We will assume that the
available information allows for the assignment to all items
within the population of a prior probability on whether or not it
possesses the rare characteristic. This is consistent with the
practice of using profiling to select high risk items for
inspection. The objective is to find the specific item with the
minimum number of inspections. We will determine the optimal
search strategies for several models according to the average
number of inspections needed to find the specific item. Using
these respective optimal strategies we show that we can order the
numbers of inspections needed for the different models partially
with respect to the usual stochastic ordering. This entails also a
partial ordering of the averages of the number of inspections.

Finally, the use, some discussion, extensions, and examples of
these results, and conclusions about them are presented.
\end{abstract}




\end{frontmatter}


\section{Introduction}
This research is motivated by several problems relevant to
security applications. Examples thereof are the search for a
terrorist among a group of passengers, for a container carrying
illicit material on a vessel entering a port, for a murderer that
has left his DNA profile at a crime scene in a small community,
etc. In general, one has to search within a finite population for
a particular item with a rare characteristic. Only close
inspection will reveal if an item possesses the characteristic or
not. Based on profiling, a relatively quick assessment is obtained
on the probability that an individual item has the rare
characteristic. Subsequently, the possibly expensive or intrusive
inspection of the high probability individuals or items is
started. The underlying idea is that this is an economically
desirable, logistically possible, and hopefully socially
acceptable way of improving security in contrast to purely random
checks or inspection of all relevant individuals.

In this research we limit ourselves to the situation where it is
certain that exactly one individual or item with the rare
characteristic belongs to the population. This situation was
studied earlier by Press \cite{press2009}. He considered a subset
of the models that we have studied in Hoogstrate and Klaassen
\cite{HoogKlaa2011} and that we study in this article. Press'
results and our results in \cite{HoogKlaa2011} are limited to the
average number of inspections, while we extend these results here
to the stochastic ordering of the numbers of inspections
themselves. Meng \cite{Meng2012} and Press \cite{press2010}
extended the results of Press \cite{press2009} from a population
at one checkpoint to a population flowing through a network of
airports with multiple checkpoints. In the present study we use
\cite{press2009} as a starting point, but we apply an axiomatic
approach, thus specifying our assumptions clearly. After
introducing our models and assumptions we discuss in subsection
\ref{comparison} similarities to and differences with
\cite{Meng2012}, \cite{press2009}, and \cite{press2010}.

\subsection{Assumptions}\ For the population the following assumptions hold.
\begin{enumerate}\label{Population}
\item {\bf Finite Population} The population consists of a finite number $N$ of
items, numbered $i=1,2,....,N.$
\item {\bf Uniqueness} One and only one of the items in the population possesses the characteristic $\Gamma$.
\item {\bf Prior Probabilities} Each item $i$ can be assigned a known probability $p_i >0$ of possessing
the characteristic $\Gamma$ we are searching for, and these
probabilities add up to 1.
\end{enumerate}
The index of the $\Gamma$-item may be viewed as the result of one
draw from the set $\{1, 2, \dots, N\}$ with sampling probabilities
$(p_1, p_2, \dots, p_N).$ We know $(p_1, p_2, \dots, p_N),$ but
not the result of the draw, which follows a multinomial
distribution with parameters 1 and $(p_1, p_2, \dots, p_N).$

\noindent
For the procedures of inspection we vary the following
assumptions.
\begin{enumerate}
\setcounter{enumi}{3}
\item {\bf Enumeration} Whether or not it is possible to enumerate and order the items according to their
associated prior probability of possessing characteristic
$\Gamma$. This translates into the issue whether or not one can
deterministically control the order in which items will be
inspected.
\item {\bf Recognition} Whether or not recognition of characteristic
$\Gamma$ is perfect. We introduce the parameter $s_i$, $0<s_i\leq
1, \, i=1,...,N,$ as the probability of recognizing characteristic
$\Gamma$ when item $i$ is inspected and actually has the
characteristic.
\item {\bf Replacement} Whether or not it is possible to apply sampling without replacement.
\item {\bf Memory} Whether or not it is possible to use the information that
an item has been selected before, and to use the outcome of this
inspection.
\end{enumerate}

\noindent Assumptions 4--7 result in 16 different models as listed in Table 1. 
Procedures are allowed only if they stop searching once the
$\Gamma$-item has been found. Formally we put the following two
conditions on the search procedures.

\begin{enumerate}\label{Procedure2}
\setcounter{enumi}{7}
\item {\bf Stopping Rule} Once the $\Gamma$-item has been found or when
no items remain for inspection, no further inspections take place.
\item {\bf Finiteness} The search procedure terminates after a finite
number of inspections.
\end{enumerate}

\noindent Next we introduce the inspection probabilities. These
are the probabilities within the models I--P that govern the
process that selects items for inspection. We note that these
probabilities are called public profile probabilities by Press
\cite{press2009}.

\begin{enumerate}
\setcounter{enumi}{9}
\item {\bf Inspection Probabilities} If an inspection takes place, the
probability that item $i$ will be inspected, is $q_i.$ We require
$\sum_{i=1}^N q_i=1$ and $q_i > 0,\, i=1,\dots, N.$
\end{enumerate}

\noindent
To enable a more detailed analysis we define the
following probabilities.
\begin{enumerate}
\setcounter{enumi}{10}
\item {\bf Attention} The sampling probability that
item $i$ comes to the attention of the inspector, is denoted by
$\lambda_i\,.$ We require $\sum_{i=1}^N \lambda_i=1$ and
$\lambda_i
> 0,\, i=1,\dots, N.$
\item {\bf Conditional Inspection} Given item $i$ has come to the
attention of the inspector, it has probability $\pi_i >0$ of being
inspected.
\end{enumerate}
Note that $q_i,\, i=1,\dots, N,$ result from the two processes
described in Assumptions 10 and 11, and that the probabilities
concerned are related by
\begin{equation}\label{profilingprobabilities}
q_i = \frac{ \lambda_i \pi_i}{\sum_{j=1}^N \lambda_j \pi_j}, \quad
i=1,\dots, N.
\end{equation}

\begin{table}\label{AllModels}
\tabcolsep=0.11cm
\begin{tabular}{|l|c|c|c|c|c|c|c|c|c|c|c|c|c|c|c|c|l|}
\hline
model index &A&B&C&D&E&F&G&H&I&J&K&L&M&N&O&P \\
\hline
enumeration & y & y & y & y & y & y & y & y & n & n & n & n & n & n & n & n\\
perfect recognition & y & y & y & y & n & n & n & n & y & y & y & y & n & n & n & n\\
with replacement & y & y & n & n & y & y & n & n & y & y & n & n & y & y & n & n\\
memory  & y & n & y & n & y & n & y & n & y & n & y & n & y & n & y & n\\
\hline
\end{tabular}
\caption{An overview of the different models: the authoritarian
models A to H, the democratic models I to P.}
\end{table}

\subsection{Discussion of the Assumptions}\
In Assumption 3 the probabilities $p_i$ are assumed to be given
without error. Of course, in practice this will often not be the
case. We will not assess the effects of uncertainty in these
probabilities by estimation here, as it is our objective to find
optimal strategies first.

Assumption 5 does not allow for false positives. We could enhance
the models by introducing a parameter representing the probability
that an item is incorrectly classified as possessing the specific
characteristic $\Gamma$ while this is actually not the case. Such
an addition is left for further research.

Note that each item $i$ with $p_i$ positive could be the
$\Gamma$-item, and hence should not be excluded from inspection
under any procedure. Exclusion would be in conflict with
assumption 9. This implies that procedures using a positive
threshold to $p_i\,,\ i=1,\dots, N,$ are excluded from our study.

Assumption 10 introduces the probabilities that govern the process
for selecting the individuals to be inspected in case enumeration
is not possible.  When it is possible to enumerate the items, one
can decide in which order the items have to be inspected. In the
models without enumeration the order in which items are inspected,
is random and depends on two processes. First it depends on the
stochastic mechanism that determines in which order items come to
the point of inspection (Assumption 11), secondly it depends on
the probability with which the item is inspected, once the item
has come to the point of inspection (Assumption 12). If some
properties or characteristics of the individuals or items in the
population are known, the resulting profiles may be used in
determining the conditional inspection probabilities $\pi_i$ or
even the sampling probabilities $\lambda_i\,.$ Obtaining an
estimate for $\pi_i$ is commonly associated with the term
profiling. The items will be inspected in an orderly sequential
fashion but the order in which items are to be inspected, might be
determined beforehand. Finally, we point out explicitly that we
assume the probabilities $p_i, q_i, s_i, \lambda_i,$ and $\pi_i$
to be constant over time and to be the same in repeated trials and
for all inspections. In practice, this assumption will often only
hold by approximation.

\subsection{Comparison between Approaches}\label{comparison}\
The most important question Press raises in \cite{press2009} is
whether actuarial methods will, from a mathematical or
probabilistic point of view, deliver the security levels as
expected by government. Subsequently he studies a stylized model
of reality and obtains both expected and surprising results. The
analyzed models however are formulated mathematically sloppily,
what makes determining their practical relevance rather difficult.


In \cite{press2010} Press analyzes the same kind of model and
optimization criterion, the average number of inspections, called
secondary checks, necessary to catch the malfeasor, but for a
network of checkpoints and under the extra constraint of allowing
for only $M$ secondary checks. As Meng \cite{Meng2012} points out
his equations (4) and (5) are wrong in that they allow
probabilities larger than 1, and Meng derives the correct
formulas.

In this setting of a maximum of $M$ secondary checks, as several
researchers, notably Meng \cite{Meng2012}, think, the optimization
criterion of minimizing the average number of checks makes hardly
any sense anymore. There is always a positive probability that the
terrorist, or malfeasor, will go through undetected. So, they
propose to optimize the probabilities $q_i$ of being selected for
inspection such as to minimize the probability of a terrorist
going through undetected. Meng \cite{Meng2012} analyzes this new
optimization criterion under the constraint on the number of
inspections and obtains some surprising results.

In our research we consider all models presented in Table
\ref{AllModels} without a maximum of $M$ secondary checks and with
the mean of the number of secondary checks as the optimization
criterion. However, in Models G, H, O, and P the probability might
be positive that the $\Gamma$-item goes through undetected. For
these models our criterion will be the conditional mean of the
number of secondary checks, given the $\Gamma$-item will be
detected. Subsequently, we order all models with their
corresponding optimal procedures, thus allowing for a balanced
decision in choosing the model and procedure appropriate for the
situation at hand. In practice our analysis is relevant when one
has a well defined closed population where there is certainty
about the existence of one item or individual having the sought
after property and it is necessary to find that person or item.

\section{Analysis}\
For each of the models we will introduce and analyze inspection
procedures. As performance measure we use the average number of
inspections that these procedures need in order to find the
$\Gamma$-item. Where possible we will minimize these averages.
Subsequently, we will study the distribution function of the
random number of inspections needed when these procedures are
applied and partially order the procedures.

\subsection{Authoritarian Models}\
In this section we analyze procedures for the models A to H from Table 1,
where enumeration and ordering of the items is
possible. We will call these models Authoritarian as in most cases
an authoritarian regime will have to be put in place in order to
get the ordering implemented, especially when the items are
people. This is in line with Press \cite{press2009}. We first
analyze the models where besides enumeration and ordering, perfect recognition
is possible.


\subsubsection{Analysis of Models A, B, C, and D}\ First we
consider model A. As we can use enumeration, ordering, and perfect
recognition, we can proceed by using the assigned prior
probabilities $p_i$ and inspect without replacement. If at the
$j$-th inspection an item with prior probability $p_{(j)}$ is
checked, then the average number of inspections for this procedure
is
\begin{equation}\label{muABCD}
\mu_{ABCD}=\sum_{j=1}^N jp_{(j)}\,.
\end{equation}
If there is an $i$ with $p_{(i)} < p_{(i+1)},$ then $\mu_{ABCD}$
can be made smaller by interchanging $p_{(i)}$ and $p_{(i+1)}.$
Consequently, as our objective is to minimize the average number
of inspections, we follow Press \cite{press2009} and choose
$p_{(1)} \geq p_{(2)} \geq \dots \geq p_{(N)},$ the ordered
probabilities $p_i.$ For the uninformative prior probabilities
$p_i=1/N,\, i=1,\dots, N,$ this yields the classical value
\begin{equation}\label{muABCD2}
\mu_{ABCD} = \frac {N+1}2.
\end{equation}
Note that under the optimal strategy with $p_{(1)} \geq p_{(2)}
\geq \dots \geq p_{(N)}$ the average $\mu_{ABCD}$ from
(\ref{muABCD}) equals at most $(N+1)/2$ from (\ref{muABCD2}). This
is an instance of Chebyshev's algebraic inequality, which may be
proved for $N$ odd by noting that
\begin{equation}
\frac {N+1}{2} - \sum_{j=1}^N jp_{(j)} =
\sum_{j=1}^N\left[\frac{N+1}2 - j\right]\ \left[p_{(j)}
-p_{\left(\frac{N+1}2 \right)} \right] \geq 0
\end{equation}
holds since each term in the second sum is nonnegative.

For the models A, B, C, and D we note that under the Stopping Rule
8 with or without memory and with or without replacement have no
effect. Consequently, the best strategies for these models are the
same, and therefore we have indicated the resulting average with
$\mu_{ABCD}.$



\subsubsection{Analysis of Models E and F}\label{analysisEF}
In Press \cite{press2009} an analysis for model E, with
enumeration and stochastic recognition, was carried out under the
reference authoritarian screening strategies with stochastic
recognition. To clarify the argument for the optimal strategy as
put forward in Press \cite{press2009}, we use conditional
probabilities. By Press' notation $s_i$ we denote the conditional
probability of identifying item $i$ at inspection as having
characteristic $\Gamma,$ given it is the $\Gamma$-item. Given that
item $j$ has been inspected $m_j$ times without having been
identified as having characteristic $\Gamma$ for $j=1, \dots, N,$
the conditional probability that item $i$ has characteristic
$\Gamma$ and will be identified at inspection, equals
\begin{equation}\label{argumentPress}
\frac{p_i (1-s_i)^{m_i}s_i}{\sum_{j=1}^N p_j
(1-s_j)^{m_j}}\,,\quad i=1,\dots, N.
\end{equation}
Consequently, in order to have the highest probability of
identifying the item with characteristic $\Gamma$ at the next
inspection, given the inspection history, one has to inspect item
$i,$ if it satisfies
\begin{equation}\label{argmax}
p_i (1-s_i)^{m_i}s_i = \max_{j=1,\dots,N} p_j (1-s_j)^{m_j}s_j \,.
\end{equation}
Note that for $s_i=1,\, i=1, \dots, N,$ (\ref{argumentPress}) and
(\ref{argmax}) present an alternative way to describe the optimal
procedure for Models A, B, C, and D. To compute the average number
of inspections under this optimal strategy we observe that the
order in which the items are to be inspected is completely
governed by (\ref{argumentPress}) and (\ref{argmax}) in a
deterministic manner as the parameters $s_i$ and $p_i$ are assumed
known. Denote the order of the items to be inspected by the
sequence of numbers $t_{ij},$ where at the $t_{ij}$-th inspection
item $i$ is inspected for the $j$-th time. The probability that
item $i$ is recognized as the $\Gamma$-item at inspection
$t_{ij},$ is
\begin{equation}\label{probofrec_ij}
p_i (1-s_i)^{j-1}s_i\,.
\end{equation}
Therefore the expected number of inspections under the optimal
strategy for model E equals
\begin{equation}\label{expectationE}
\mu_{EF}=\sum_{i=1}^N\sum_{j=1}^\infty t_{ij} p_i
(1-s_i)^{j-1}s_i\,.
\end{equation}
Note that in case of perfect recognition (\ref{expectationE})
reduces to
\begin{equation}\label{expectationEreduced}
\mu_{ABCD}=\sum_{i=1}^N\sum_{j=1}^\infty t_{ij} p_i 0^{j-1} =
\sum_{i=1}^N t_{i1}p_i = \sum_{i=1}^N t_{i1}p_{(t_{i1})} =
\sum_{j=1}^N jp_{(j)} .
\end{equation}
 To verify that (\ref{expectationE}) is optimal indeed,
we compare time $t_{ij}$ with $t_{ij}+1=t_{k\ell}$. If $k=i$ then
$\ell=j+1$ and we do nothing. However, if $k\neq i$ then reversing
the order of these two inspections gives a smaller value for
$\mu_{EF}$ if and only if
\begin{equation}\label{condition} p_i (1-s_i)^{j-1}s_i -
p_k (1-s_k)^{\ell-1}s_k <0.
\end{equation}
This implies that at each point in time (\ref{argmax}) should
hold.

Note that [12] of Press \cite{press2009} is equivalent to
(\ref{expectationE}) and that [13] of Press \cite{press2009}
follows by $P(N_{EF}=t_{ij})=p_i(1-s_i)^{j-1}s_i$ and
\begin{equation}\label{noref}
\sum_{i=1}^N \sum_{j=1}^\infty P(N_{EF}=t_{ij}) =
 \sum_{i=1}^N \sum_{j=1}^\infty p_i (1-s_i)^{j-1} s_i =1.
\end{equation}
Here $N_{EF}$ is defined as the number of inspections needed to
find the $\Gamma$-item under the optimal strategy.

For the analysis of procedure $F$ we just observe that as the
order in which the items are inspected can be determined in
advance just as in model E, the optimal strategy and subsequent
analysis are the same for model E and F, whence the notation
$\mu_{EF}$ and $N_{EF}$.

\subsubsection{Analysis of Models G and H}\

Models G and H satisfy the same conditions as Models C and D,
except for the perfect recognition condition. In fact, Models G
and H are a generalization of Models C and D, respectively, in the
sense that for $s_i=1,\, i=1, \dots, N,$ these models are the
same. In these four models there is sampling without replacement,
and hence in Models G and H there is a possibility that the
$\Gamma$-item will not be found. Indeed, the probability the
$\Gamma$-item will be found equals $\sum_{i=1}^N s_i p_i$ here,
and if this probability is less than 1, the distribution of the
number of inspections needed to identify the $\Gamma$-item is
defective. In this case it makes sense to take the number of
inspections as infinity if the $\Gamma$-item has not been
identified, and consequently we then have
\begin{equation}
\mu_{GH}=\infty.
\end{equation}
One might be interested in the conditional expectation of the
number of inspections given the $\Gamma$-item will be found. This
conditional expectation equals $\mu_{ABCD}$ given in
(\ref{muABCD}).

Since the probability that the $\Gamma$-item will not be found,
equals $\sum_{i=1}^N (1-s_i) p_i$ and does not depend on the
choice of the $q_i$s, we define the optimal strategy as the same
one that minimizes $\mu_{ABCD}$ given in (\ref{muABCD}). Hence it
makes sense to write
\begin{equation}\label{muGH}
\mu_{GH}=\sum_{i=1}^N s_i p_i \sum_{j=1}^N jp_{(j)} + \left(1-
\sum_{i=1}^N s_i p_i \right) \infty
\end{equation}
with $0 \times \infty$ interpreted as 0.



\subsection{Democratic Models}\
In this section we analyze the six models I to N. We note that
models J and N have been analyzed by Press \cite{press2009} as
democratic strategies under perfect recognition and stochastic
recognition, respectively.

\subsubsection{Analysis of Models I, K, and L}\
If within the models I, K, or L an inspected item is not the
$\Gamma$-item, it is not selected again for inspection either
because the model is without replacement (K and L) or because the
item is being recognized as having had a negative outcome of the
inspection before (I). Consequently, these models give rise to the
same optimal procedure.


By $(I_1, I_2, \dots, I_N)$ we denote the random vector of indices
that describes in which order the items in the population will be
inspected. This random vector is ruled by the $q_i$ from
Assumption 10, to wit
\begin{equation}\label{IKL1}
P\left(I_1 = i_1, \dots, I_N = i_N \right) =
\prod_{j=1}^N\frac{q_{i_j}}{1-\sum_{h=1}^{j-1}q_{i_h}}.
\end{equation}
The conditional expectation of the number of inspections needed,
given $(I_1, I_2, \dots, I_N)=(i_1,i_2, \dots, i_N),$ equals
\begin{equation}\label{conditionalAverage}
\sum_{k=1}^N k p_{i_k}\,.
\end{equation}
Consequently, the unconditional average number of inspections
equals
\begin{equation}\label{muiklProb}
\mu_{IKL}=\sum_{(i_1,...,i_N)} \sum_{k=1}^N k p_{i_k}
\prod_{j=1}^N\frac{q_{i_j}}{1-\sum_{h=1}^{j-1}q_{i_h}},
\end{equation}
where the first summation is over the collection of $N!$ vectors
$(i_1,...,i_N)$ that can be obtained by permutation of $(1, \dots,
N)$. Calculation of the optimal $q_1,...,q_N$ for this model is an
extremely daunting task. However, we note that the special case of
the uniform distribution with $q_i=1/N,\, i=1,\dots, N,$ yields
$\mu_{IKL}= (N+1)/2,$ which is no surprise.



\subsubsection{Analysis of Model J}\
In Model J we have sampling with replacement, actually. Let $C$ be
the index of the item that has characteristic $\Gamma$ and let $T$
be the number of inspections needed to identify the item with
characteristic $\Gamma.$ Note that $C$ is random with distribution
$P(C=i)=p_i\,,\,i=1,\dots, N,$ according to Assumption 2. Given
$C=i,$ the random variable $T$ has a geometric distribution with
success probability $q_i\,,$ i.e.
\begin{equation}\label{geometric}
P(T=j\,|\,C=i)= (1-q_i)^{j-1} q_i\,,
\end{equation}
and mean
\begin{equation}\label{conditionalmean}
\sum_{j=1}^\infty j (1-q_i)^{j-1} q_i = \frac 1{q_i}.
\end{equation}
Consequently, the average number $\mu_J$ of inspections needed is
the expectation of (\ref{conditionalmean}) and equals (cf. [3] of
Press \cite{press2009})
\begin{equation}\label{DM}
\mu_J=\sum_{i=1}^N \frac{p_i}{q_i}.
\end{equation}
By the Cauchy-Schwarz inequality we have
\begin{equation}\label{CS}
\left(\sum_{i=1}^N \sqrt{p_i}\right)^2 = \left(\sum_{i=1}^N \frac
{\sqrt{p_i}} {\sqrt{q_i}} \sqrt{q_i} \right)^2 \leq \sum_{i=1}^N
\frac{p_i}{q_i} \sum_{j=1}^N q_j = \mu_J
\end{equation}
with equality if and only if
\begin{equation}\label{equality}
q_i = \frac {\sqrt{p_i}}{\sum_{j=1}^N \sqrt{p_j}},\quad i=1,\dots,
N,
\end{equation}
holds. Note that (\ref{equality}) yields the optimal strategy.



\subsubsection{Analysis of Models M and N}\
In model N the same conditions hold as in model J. However, there
is no perfect recognition. Instead we assume that the probability
of recognizing item $i$ as the $\Gamma$-item when it is inspected,
is given by the probability $s_{i}.$ Based on the same reasoning
as in (\ref{geometric}) and (\ref{conditionalmean}) we get the
average number of required inspections, given $C$, as $1/(q_C
s_C).$ Taking the expectation over this $C$, we obtain (cf.
(\ref{DM}))
\begin{equation}
\mu_{MN}=\sum_{i=1}^N \frac{p_{i}}{q_i s_i}.
\end{equation}
Minimization as in (\ref{CS}) and (\ref{equality}) shows that the
optimal strategy and minimized value of $\mu_{MN}$ are given by
\begin{equation}
q_i = \frac{\sqrt{p_i / s_i}}{\sum_{j=1}^N \sqrt{p_j / s_j}},
\quad \mu_{MN} = \left(  \sum_{i=1}^N \sqrt{\frac{p_i}{s_i}}
\right)^2.
\end{equation}

For model M we obtain the same results as for model N. At first
sight this is a bit strange, but it is due to the fact that the
optimal strategy is obtained using known $p_i$ and $s_i$. That
means that remembering whether somebody has been screened already
and found not to be the item with characteristic $\Gamma,$ does
not give additional information and consequently the optimal
strategy for model N cannot be improved within model M. However,
in practice the values of $p_i$ and $s_i$ would have to be
estimated and obtaining a negative observation would mean an
adjustment in the estimates for $p_i$ and $s_i$.


\subsubsection{Analysis of Models O and P}\

The relationship between Models O and P on the one hand and Models
K and L on the other hand is the same as between Models G and H
and Models C and D, respectively. They differ in the perfect
recognition condition. In these models there is sampling without
replacement, and the probability the $\Gamma$-item will be found
equals $\sum_{i=1}^N s_i p_i\,.$ If this probability is less than
1, it makes sense to take the number of inspections as infinity if
the $\Gamma$-item has not been identified.

Like for Models G and H, we define the optimal strategy as the
same one that minimizes $\mu_{IKL}$ given in (\ref{muiklProb}),
and we write
\begin{equation}\label{muOP}
\mu_{OP}=\sum_{i=1}^N s_i p_i \sum_{(i_1,...,i_N)} \sum_{k=1}^N k
p_{i_k} \prod_{j=1}^N\frac{q_{i_j}}{1-\sum_{h=1}^{j-1}q_{i_h}} +
\left(1- \sum_{i=1}^N s_i p_i \right) \infty.
\end{equation}



\section{Ordering of Models}
In the above analysis we introduced several procedures minimizing
the average number of inspections necessary to find the
$\Gamma$-item under varying assumptions on the investigative
environment. This left us with the averages $$\mu_{ABCD}\,,\
\mu_{EF}\,,\ \mu_{GH}\,,\ \mu_{IKL}\,,\ \mu_J\,,\ \mu_{MN}\,,\
\mu_{OP}\,,$$ which we try to put in increasing order in this
section. If one is smaller than the other an investigator could
try to change the conditions under which one has to conduct the
investigation such that the assumptions of the procedure with the
smaller average number of inspections can be met.

Denote the number of inspections needed by the optimal strategy
for each of the models discussed above by
$$N_{ABCD}\,,\ N_{EF}\,,\ N_{GH}\,,\ N_{IKL}\,,\ N_J\,,\
N_{MN}\,,\ N_{OP}\,,$$ where the subscripts denote the model.
These numbers are random variables, which can be ordered
partially. We need the following definition.

\begin{definition}\label{stoch}
Random variable $X$ is stochastically smaller than random variable
$Y$, if and only if $$P(X\leq z)\geq P(Y\leq z)$$ holds for all
$z\in \mathbb{R};$ notation $X\preceq_{st}Y.$
\end{definition}

It is clear that not any pair of random variables can be ordered
stochastically, but some of the above numbers of inspections can,
as stated in our theorem.

\begin{theorem}\label{OverallOrderNumber}
Under Assumptions 1--10 the optimal strategies for models A--P as
defined and analyzed above, give rise to the following partial
ordering of the corresponding random numbers of inspections
\begin{align}
    N_{ABCD}&\preceq_{st} N_{EF}\preceq_{st}
    N_{MN}\,, \label{OONvlg1} \\
    N_{ABCD}&\preceq_{st} N_{GH}\preceq_{st} N_{OP}\,, \label{OONvlg2} \\
    N_{ABCD}&\preceq_{st} N_{IKL}\preceq_{st} N_J\preceq_{st}
    N_{MN}\,, \label{OONvlg3} \\
    N_{IKL}&\preceq_{st} N_{OP}\,. \label{OONvlg4}
\end{align}
Furthermore, the first inequality of (\ref{OONvlg1}) and of
(\ref{OONvlg2}), the last inequality of (\ref{OONvlg3}), and
(\ref{OONvlg4}) are equalities if and only if $\sum_{i=1}^N s_i
p_i=1$ holds. The second inequality of (\ref{OONvlg2}) and the
first inequality of (\ref{OONvlg3}) are equalities if and only if
$p_1=\dots = p_N=1/N$ holds. The second inequality of
(\ref{OONvlg1}) and of (\ref{OONvlg3}) are equalities if and only
if $N=1$ holds.

Finally, $N_{EF}$ is not comparable to $N_{GH},$ nor to $N_{IKL},
N_J,$ and $N_{OP}$ in this stochastic ordering, $N_{GH}$ is not
comparable to $N_{IKL}, N_J,$ and $N_{NM},$ and $N_{OP}$ is not
comparable to $N_J$ and $N_{MN}.$
\end{theorem}

This result has its consequences for the average numbers of inspections.

\begin{corollary}\label{OverallOrderAverage}
Under the Assumptions 1--10 the models A--P as analyzed above give
rise to the following ordering of the average numbers of
inspections corresponding to the optimal strategies
\begin{align*}
    \mu_{ABCD}&\leq\mu_{EF}\leq\mu_{MN}\,, \\
    \mu_{ABCD}&\leq\mu_{GH} \leq \mu_{OP}\,, \\
    \mu_{ABCD}&\leq\mu_{IKL} \leq\mu_J\leq\mu_{MN}\,, \\
    \mu_{IKL}&\leq\mu_{OP}\,.
\end{align*}
\end{corollary}

\begin{proof}[Proof of Corollary \ref{OverallOrderAverage}]
Just note that for nonnegative random variables $X$ the
expectation ${\mathbb E}X$ equals $${\mathbb E}X = \int_0^\infty
P(X>x) dx.$$
\end{proof}

\begin{figure}
\caption{\label{t1graphically} A graphical representation of Theorem \ref{OverallOrderNumber}.}
\begin{tikzpicture}[->,>=stealth',shorten >=1pt,auto,node distance=3cm,
  thick,main node/.style={circle,fill=blue!20,draw,font=\sffamily\bfseries},minimum size=12mm]

  \node[main node] (abcd) {ABCD};
  \node[main node] (ikl) [right of = abcd] {IKL};
  \node[main node] (ef) [above of = ikl] {EF};
  \node[main node] (mn) [right of = ef] {MN};
  \node[main node] (j) [right of = ikl] {J};
  \node[main node] (gh) [below of = ikl] {GH};
  \node[main node] (op) [right of = gh] {OP};

  \path[every node/.style={font=\sffamily\small},ultra thick]
    (abcd) edge node [left] {} (ef)
        edge node [left] {} (ikl)
        edge node [left] {} (gh)
    (ef) edge node [left] {} (mn)
    (ikl) edge node [left] {} (j)
        edge node [left] {} (op)
    (j) edge node [left] {} (mn)
    (gh) edge node [left] {} (op);
\end{tikzpicture}
\end{figure}
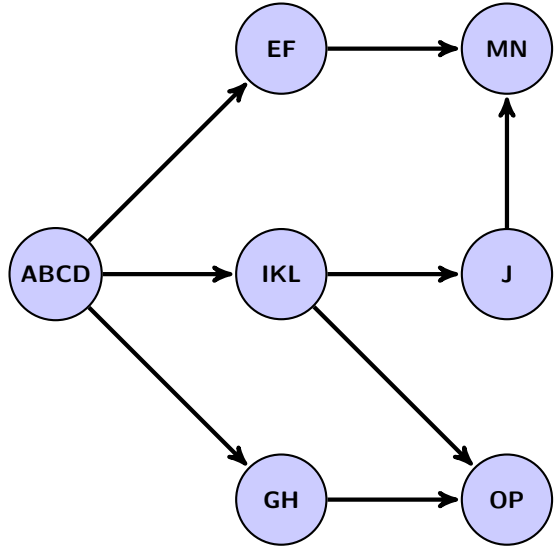

The results of Theorem \ref{OverallOrderNumber} can also be displayed graphically.
The result is depicted in Figure \ref{t1graphically}. The direction of the arrows between two nodes indicates the ordering between the two associated models. Models that are not comparable are not connected.

\begin{proof}[Proof of Theorem \ref{OverallOrderNumber}]
Let, as in Subsection \ref{analysisEF}, item $i$ be inspected for
the $j$-th time at the $ t_{ij}$-th inspection under model EF. To
prove the first inequality of (\ref{OONvlg1}) we just note that
for any positive integer $m$
\begin{eqnarray}\label{NEFNABCD}
\lefteqn{P\left(N_{EF}\leq m \right) = \sum_{i=1}^N p_i
\sum_{j=1}^\infty {\bf 1}_{[ t_{ij} \leq m ]} (1-s_i)^{j-1} s_i \nonumber} \\
& & = \sum_{i=1}^N p_i {\bf 1}_{[ t_{i1} \leq m ]}
\sum_{j=1}^\infty {\bf 1}_{[ t_{ij} \leq m ]} (1-s_i)^{j-1} s_i \leq  \sum_{i=1}^N p_i {\bf 1}_{[ t_{i1} \leq m ]} \\
& & \leq \sum_{j=1}^{\min\{m,N\}} p_{(j)} = P(N_{ABCD}\leq m)
\nonumber
\end{eqnarray}
holds with $p_{(1)} \geq p_{(2)} \geq \dots \geq p_{(N)}.$ Note
that equalities hold here if and only if $s_i$ equals 1 whenever
$p_i$ is positive, i.e. if and only if $\sum_{i=1}^N s_i p_i=1$
holds.

Note that under model MN the optimal strategy for model EF as
described in (\ref{argumentPress}) and (\ref{argmax}) cannot be
applied since there is no enumeration. This coupling argument
shows the second inequality of (\ref{OONvlg1}), which reduces to
an equality if and only if models EF and MN coincide, i.e. for
$N=1.$

One might call $N_{GH}$ a defective version of $N_{ABCD}$ in that
inspection proceeds in exactly the same way, unless because of
imperfect recognition the $\Gamma$-item has not been recognized
and hence will never be, in which case $N_{GH}=\infty$ holds. This
proves the first inequality of (\ref{OONvlg2}) with equality if
and only if $\sum_{i=1}^N s_i p_i=1$ holds. Similarly, inequality
(\ref{OONvlg4}) and its equality condition are proved.

The first inequality from (\ref{OONvlg3}) and its condition for
equality are proved in Lemma \ref{ABCDsmallerIKL} of the Appendix.
Since $N_{GH}$ and $N_{OP}$ are defective versions of $N_{ABCD}$
and $N_{IKL},$ respectively, this Lemma also proves the second
inequality of (\ref{OONvlg2}) and its equality condition.

Note that under model IKL items that have been checked before, are
not checked again, either because the perfect memory is used or
because items that have been checked, are not replaced. Under J
these items can still be sampled. This coupling argument proves
the second inequality from (\ref{OONvlg3}). Perfect, restricted,
or no memory and with or without replacement do not make a
difference here if and only if $N=1.$

Finally, the third ordering relation of (\ref{OONvlg3}) and its
equality condition are proved by Lemma \ref{JsmallerMN} of the
Appendix.

12 out of the $\binom{7}{2}=21$ possible pairs of numbers of
inspections have been ordered stochastically in
(\ref{OONvlg1})--(\ref{OONvlg4}). The other 9 pairs cannot be
stochastically ordered, as we will show now. Let $B$ be a
Bernoulli random variable with $P(B=1)=1-P(B=0)= \sum_{i=1}^N s_i
p_i$ that is independent of all random numbers of inspections.
Note that $N_{GH}$ has the same distribution as the defective
random variable $B N_{ABCD} +(1-B)\infty.$ Since inequality
(\ref{OONvlg1}) and the inequalities of (\ref{OONvlg3}) can be
strict, this representation of $N_{GH}$ shows that it is not
comparable to $N_{EF}, N_{IKL}, N_J,$ and $N_{MN}$ stochastically.
By an analogous argument $N_{OP}= B N_{IKL} +(1-B)\infty$ is not
comparable to $N_J$ and $N_{MN}.$

As Lemma \ref{EFincomparable} of the Appendix shows, $N_{EF}$ is
not comparable to $N_{OP}$ nor to $N_{IKL}$ and $N_J.$
\end{proof}


\section{Further Explorations}

\subsection{Profiling}\
The interesting question for the models studied here is how an
improvement of the differentiating power of the prior
probabilities $p_i$ affects the efficiency. Good discriminatory
prior probabilities lead to a limited group of relatively high
probability items and a large group of small probability items.
Obviously the optimal situation across all models is a degenerate
prior probability distribution with probability 1 for the
$\Gamma$-item and 0 for the other items. The closer one gets to
this distribution the better. In practice the prior probabilities
consist of probabilities based on available information.
Estimating these probabilities using the available information is
often referred to as profiling. The results here open up the
possibility to try to balance the costs of collecting additional
data in order to improve the discriminatory power of profiling and
the costs of additional inspections needed to find the
$\Gamma$-item.

\subsection{Updating Prior Probabilities}\
Recall that the random variable $C$ denotes the index of the
$\Gamma$-item. Furthermore, assume information $I$ has come to the
attention of the investigators before the procedure has started.
We may update the prior probability $p_i=P(C=i),$ using Bayes
rule, by
\begin{equation}\label{BayesRule}
P(C=i\,|\,I)=\frac{P(I\,|\, C=i)p_i}{\sum_{j=1}^N
P(I\,|\,C=j)p_j}.
\end{equation}


\subsection{Inspection and Profiling Probabilities}\
Let us substitute the inspection probabilities from Assumption 10
by the profiling probabilities from Assumptions 11 and 12 as in
(\ref{profilingprobabilities}). If we do this in the derived
expressions for the optimal inspection strategies in
(\ref{equality}) for model J, we see that the conditional
inspection probabilities $\pi_i$ have to satisfy
\begin{equation}\label{equationpi}
\pi_i=\frac{\sqrt{p_i}/\lambda_i}{\sum_{j=1}^N \sqrt{p_j}}\,
\sum_{h=1}^N \lambda_h \pi_h\,,\quad i=1,\dots, N.
\end{equation}
Note that these equations determine the $\pi_i$ up to a constant.
Consequently, if $\pi_1,\dots, \pi_N$ are optimal, so are
$c\pi_1,\dots, c\pi_N$ for any $0<c\leq 1/\max_{1 \leq j \leq N}
\pi_j .$ This argument shows that the intensity of actual
inspections may be chosen quite freely without influencing $N_J$
or $\mu_J.$ The same reasoning holds for $N_{MN}$ and $\mu_{MN}.$

Finally, we would like to note that (\ref{profilingprobabilities})
also shows that if for some $i$ the attention probability
$\lambda_i$ is relatively high, then the conditional inspection
probability $\pi_i$ should be chosen relatively small.
Consequently, items that are more likely to come to the attention
of the inspectors, like for instance frequent flyers at airports,
should have a smaller probability to be inspected.

\subsection{Related and future research}\
Besides the research of Press \cite{press2009} the topic of the
present paper has not received much attention yet. In adjacent
fields of research more results are available. Our results however
are more or less complementary to these results. Boland et al.
\cite{Boland2002}, \cite{Boland2003}, \cite{Boland2004} in a
series of papers study stochastic orders of partition and random
testing for faults in software. Some of our models, specifically
the models with enumeration, are a limiting case when a partition
can contain a single item. Montanaro \cite{montanaro2009} shows
that information about items in an unstructured but enumerated
list can speed up the search for a single item in quantum search
relative to quantum search using no additional information.

Further research should include the situation of an unknown,
random number of items with the rare characteristic. These models
are highly relevant when optimizing security screening
applications. The same holds for analyzing the effects of
estimated prior probabilities, or replacing them with conditional
probabilities.


\section{Examples}
To show the relevance and the use of our results we give some
examples and some guidance on how to act in some practical cases.
The goal is to find the item with the $\Gamma$-characteristic as
efficiently as possible. In practice all kinds of situations can
occur and our intention is to choose the best model to handle the
situation. Further applications like fraud detection etc., are left to the imagination of the reader.

\subsection{Example: DNA Screening}\
In a small village a murder has taken place. Due to the isolated
nature of the village and some other indications the police
strongly belief that the murderer is one of the men in the
village. The police suggests requesting for a DNA analysis of all
these men as there was DNA found at the scene of the crime.
However, DNA analyses take time and money, and they compromise the
privacy of the people involved. So, the strategy should be to take
as few DNA samples as possible. What is the optimal available
strategy, assuming that the perpetrator is among the male
inhabitants? The police can assign prior probabilities to the
men that indicate how likely it is that they are the murderer.
Furthermore, they can also enumerate the men and order them
according to the assigned probabilities. So, in this case the best
choice is to go for model ABCD, since a DNA match might be
considered as perfect recognition here.

\subsection{Example: Customs}\
Customs has to check containers transported by sea for illicit materials like drugs and
after 9/11 also for nuclear materials, weapons, explosives, and
biological and chemical weapons. Every once and a while customs
gets credible information that a container at a certain ship
contains one of these illegal materials. Suppose the ship is
carrying $N=5000$ containers, and that for each of them a risk
profile is available so that one can assign prior probabilities
$p_i$ of containing the illicit material. How to check these
containers as efficiently as possible? If one could first
completely unload the ship and set the containers on the dock,
then one could use model ABCD if the recognition of the illicit
material was perfect. In case of imperfect recognition one would
use model E. If one does not have this possibility, but only has
the possibility to decide whether to check or not when the
container leaves the ship, the preferred model under perfect
recognition would be IKL, and model MN otherwise.

\subsection{Example: Entrance}\
The authorities have received information that a certain criminal
is trying to escape from the country with a false identity using a
certain plane. How to proceed most efficiently? That is, not
disrupting the flight schedule and not aggravating innocent
public. In this case if one assumes that recognition is perfect,
one could use model ABCD if one calls the passengers one by one.
Procedures can resume to normal once the criminal has been found.
If one assumes recognition not to be perfect, model EF could be
chosen. Note that everybody has to stay at the gate until the
criminal has been identified, and some individuals will possibly
have to be inspected multiple times. Observe that inspecting
individuals just as they enter, i.e. in random order, will be less
efficient.

\subsection{Example: Robber in Town}\
Suppose a bank is robbed and the robbers get-away car is red.
Which model applies? Model J or model IKL? The difference?
Assuming that every car had an initial probability of being used
for a robbery, we can now update these prior probabilities using
(\ref{BayesRule}) to incorporate the information that the car is
red. When there is perfect recognition one could go for model IKL
but more likely for model J. In the case of stochastic recognition
one resides in model MN. The use of model IKL in practice would
need communication and coordination between different policing
units.



\section{Conclusions}
\noindent In this paper we introduced a framework of models that
can be used to analyze how to find an item of interest in a finite
population when the probability of an item of actually being the
item of interest is assumed known or partially known based on
prior information. The results can be used in several ways. First
they can give investigators and developers of security protocols
ideas on how to design the physical inspections. Secondly they can
give a first handle on weighing the costs of improving prior
information against reduction in costs of more thorough
inspections. Furthermore, the results extend the results of Press
\cite{press2009} in the sense that not only the averages of the
numbers of needed inspections for the different models are ordered
but also these random numbers themselves. Finally, for the
democratic case of model J our results imply that people often
travelling and therefore having a larger probability of coming to
the attention of the inspectors should receive a lower conditional
probability of actually getting an inspection according to the
optimal inspection strategy.


\section{Appendix: Three Lemmata}


\begin{lemma}\label{ABCDsmallerIKL}
$N_{ABCD}$ is stochastically smaller than \ $N_{IKL},$ i.e.
$P(N_{ABCD} > m) \leq P(N_{IKL} > m)$ holds for all positive
integers $m,$ with equality if and only if $p_1 = \dots = p_N=1/N$
holds.
\end{lemma}

\medskip\noindent
\begin{proof}
First consider, with $C$ the index of the $\Gamma$-item,

\begin{equation}\label{sums}
\lefteqn{}P(N_{IKL} > m \,|\, C=k)=  \sum_{\substack{i_1=1\\i_1\neq k}}^N f_{i_1}
        \sum_{\substack{i_2=1\\i_2\neq k\\i_2\neq i_1}}^Nf_{i_2}
        \sum_{\substack{i_3=1\\i_3\neq k\\i_3\neq i_1\\i_3\neq i_2}}^Nf_{i_3}
         \cdots \sum_{\substack{i_m=1\\ i_m\neq k\\ i_m\neq i_1\\ i_m\neq i_2
        \\ \cdots\\ i_m\neq
        i_{m-1}}}^Nf_{i_m},
\end{equation}
with
\begin{equation}\label{factor}
f_{i_j}=\frac{q_{i_j}}{1-q_{i_1}-\cdots - q_{i_{j-1}}}
\end{equation}
for $j=2,\dots,N$ and $f_{i_1}=q_{i_1}$.

Since the $q_i$'s add up to 1, the last sum in (\ref{sums}) equals
\begin{equation}\label{simple}
1-\frac{q_k}{1-q_{i_1}-q_{i_2}-\cdots q_{i_{m-1}}}.
\end{equation}
Addition of (\ref{sums}) over $k$ from 1 to $\ell$ taking into
account (\ref{simple}) yields
\begin{eqnarray}\label{sums2}
\lefteqn{\sum_{k=1}^\ell P( N_{IKL} > m \,|\, C=k) \nonumber} \\
& & = \sum_{i_1=1}^N
        \sum_{\substack{i_2=1\\i_2\neq i_1}}^N
        \sum_{\substack{i_3=1\\i_3\neq i_1\\i_3\neq i_2}}^N
        \cdots
        \sum_{\substack{i_{m-1}=1\\i_{m-1}\neq i_1\\i_{m-1}\neq i_2\\\cdots\\i_{m-1}\neq i_{m-2}}}^N
        f_{i_1}f_{i_2}f_{i_3} \cdots f_{i_{m-1}} \nonumber \\
& & \qquad \qquad \qquad \qquad \qquad \qquad\sum_{\substack{k=1\\k\neq i_j\\
j=1,\dots,m-1}}^\ell \left( 1-\frac{q_k}{1-q_{i_1}-q_{i_2}-\cdots
      q_{i_{m-1}}}\right),
\end{eqnarray}
where we have interchanged the summation over $k$ with the
summations over the $i_j$'s. Note that the last sum in
(\ref{sums2}) equals at least
\begin{equation}\label{bound}
\sum_{\substack{k=1\\k\neq i_j\\
j=1, \dots,m-1}}^\ell 1  -  \sum_{\substack{k=1\\k\neq i_j\\
j=1,\dots,m-1}}^N \frac{q_k}{1-q_{i_1}-q_{i_2}-\cdots
      q_{i_{m-1}}}  \geq \ell -(m-1) - 1= \ell -m.
\end{equation}
Combining (\ref{sums2}) and (\ref{bound}) we obtain
\begin{eqnarray}\label{sums3}
\lefteqn{\sum_{k=1}^\ell P( N_{IKL} > m \,|\, C=k) \nonumber} \\
& &    \geq \sum_{i_1=1}^N
        \sum_{\substack{i_2=1\\i_2\neq i_1}}^N
        \sum_{\substack{i_3=1\\i_3\neq i_1\\i_3\neq i_2}}^N
        \cdots
        \sum_{\substack{i_{m-1}=1\\i_{m-1}\neq i_1\\i_{m-1}\neq i_2\\\cdots\\i_{m-1}\neq i_{m-2}}}^N
        f_{i_1}f_{i_2}f_{i_3}
        \cdots f_{i_{m-1}}\,(\ell -m) \nonumber \\
& &       =  \ell - m.
\end{eqnarray}
Without loss of generality we may assume that the items in the
population have been numbered such that
\begin{equation}\label{ordered}
p_1 \geq p_2 \geq \dots \geq p_N\,.
\end{equation}
Then, we have
\begin{equation}\label{authoritarian}
P(N_{ABCD}> m)= \sum_{k=m+1}^N P(C=k) = \sum_{k=m+1}^N p_k\,.
\end{equation}
Together with (\ref{sums3}) and (\ref{ordered}) this equality
yields, with $p_{N+1}=0,$
\begin{eqnarray}\label{lowerbound}
\lefteqn{P(N_{IKL} > m) - P(N_{ABCD} > m) =
\sum_{k=1}^N \left\{P(N_{IKL}>m\,|\,C=k)-{\bf 1}_{[k>m]}\right\}p_k \nonumber }\\
& & \qquad \qquad \qquad =\sum_{k=1}^N \sum_{\ell=k}^N \left\{P(N_{IKL}>m\,|\,C=k)-{\bf 1}_{[k>m]}\right\} \left(p_\ell-p_{\ell+1}\right) \nonumber \\
& & \qquad \qquad \qquad =\sum_{\ell=1}^N \sum_{k=1}^\ell
\left(p_\ell-p_{\ell+1}\right)\left\{P(N_{IKL}>m\,|\,C=k)-{\bf
1}_{[k>m]}\right\} \nonumber \\
& & \qquad \qquad \qquad \geq  \sum_{\ell=1}^N
\left(p_\ell-p_{\ell+1}\right)\left\{[\ell - m]^+ -
\sum_{k=1}^\ell {\bf 1}_{[k>m]}\right\} = 0,
\end{eqnarray}
where $[x]^+$ equals the maximum of $x$ and 0.
Inequality (\ref{lowerbound}) proves the stochastic ordering. (\ref{lowerbound}) and
(\ref{sums2})--(\ref{sums3}) show that $N_{ABCD}$ and $N_{IKL}$ have
the same distribution if and only if $p_1 = \dots = p_N$ holds.
\end{proof}



\begin{lemma}\label{JsmallerMN}
$N_J$ is stochastically smaller than $N_{MN},$ i.e. $P(N_J \leq m)
\geq P(N_{MN} \leq m)$ holds for all positive integers $m,$ with
equality if and only if $\sum_{i=1}^N s_i p_i =1$ holds.
\end{lemma}
\begin{proof}
We have
\begin{eqnarray}\label{BasicDemocraticForm}
\lefteqn{P(N_J \leq m)= \sum_{k=1}^N
P(N_J \leq m\,|\,C=k) P(C=k) \nonumber} \\
& & \qquad =\sum_{k=1}^N \sum_{\ell=1}^m (1-q_{k})^{\ell-1}q_k p_k
= \sum_{k=1}^N \frac{1-(1-q_k)^m}{1-(1-q_k)}\, q_k p_k \\
& & \qquad = \sum_{k=1}^N [1-(1-q_k)^m]p_k = 1-\sum_{k=1}^N
(1-q_k)^m p_k\,. \nonumber
\end{eqnarray}
Similarly, for $N_{MN}$ we have
\begin{equation}\label{dfNMN}
P(N_{MN}\leq m)=1-\sum_{k=1}^N (1-s_kq_k)^m p_k\,.
\end{equation}
This leads to
\begin{equation}
P(N_J \leq m)- P(N_{MN} \leq m) = \sum_{k=1}^N \left[ (1-s_k
q_k)^m - (1- q_k)^m \right]\, p_k \geq 0 \nonumber
\end{equation}
in view of $0 < s_k \leq 1,$ which proves the lemma.
\end{proof}



\begin{lemma}\label{EFincomparable}
$N_{EF}$ is stochastically not comparable to $N_{OP}$ nor to
$N_{IKL}$ and $N_J.$
\end{lemma}
\begin{proof}
If both $\sum_{i-1}^N s_i p_i=1$ and $N>1$ hold and not all $p_i$
are equal to $1/N,$ then Theorem \ref{OverallOrderNumber} and
Corollary \ref{OverallOrderAverage} imply
\begin{equation}\label{less}
\mu_{EF}=\mu_{ABCD} < \mu_{IKL} < \mu_J.
\end{equation}
However, for $N=1$ and $s_1<1$ we have
\begin{equation}\label{more}
\mu_{EF}= \frac 1{s_1} > 1 = \mu_{IKL}= \mu_J.
\end{equation}
Inequalities (\ref{less}) and (\ref{more}) show that $N_{EF}$
cannot be stochastically ordered with respect to $N_{IKL}$ and
$N_J$ without additional conditions.

Comparing $N_{EF}$ to $N_{OP}$ we choose
\begin{equation}\label{choices}
p_i = \frac{2i}{N(N+1)},\quad s_i = \frac 1i, \quad i=1, \dots, N.
\end{equation}
Now, on the one hand
\begin{equation}\label{infinity}
P\left( N_{EF} < \infty \right) =1, \quad P\left(N_{OP} = \infty
\right) = 1-\frac 2{N+1}
\end{equation}
holds and on the other hand (\ref{argmax}) implies
\begin{equation}\label{EF1}
P\left( N_{EF}=1 \right) = \max_{1 \leq j \leq N} s_j p_j = \frac
2{N(N+1)}
\end{equation}
and the second inequality of (\ref{OONvlg3}),
(\ref{BasicDemocraticForm}), and (\ref{equality}) yield
\begin{eqnarray}\label{OP1}
\lefteqn{P\left( N_{OP} = 1 \right) = P\left( B N_{IKL} +
(1-B)\infty = 1 \right)\nonumber} \\
& & \geq P(B=1) P\left( N_J = 1 \right) = P(B=1) \sum_{i=1}^N q_i
p_i \nonumber \\
& & = \sum_{i=1}^N s_i p_i \frac{\sum_{i=1}^N p_i
\sqrt{p_i}}{\sum_{i=1}^N \sqrt{p_i}}= \frac 4{N(N+1)^2}
\frac{\sum_{i=1}^N i \sqrt{i}}{\sum_{i=1}^N \sqrt{i}}.
\end{eqnarray}
The right hand side of (\ref{OP1}) is larger than the right hand
side of (\ref{EF1}) if and only if
\begin{equation}\label{averages1}
\frac{\sum_{i=1}^N i \sqrt{i}}{\sum_{i=1}^N \sqrt{i}} >
\frac{N+1}2
\end{equation}
if and only if
\begin{equation}\label{averages2}
\sum_{i=1}^N \left(i -\frac{N+1}2 \right) \left( \sqrt{i}
-\sqrt{\frac{N+1}2}\right) > 0,
\end{equation}
which holds for $N \geq 2;$ Chebyshev's algebraic inequality. This
shows
\begin{equation}
P\left( N_{OP} = 1 \right) > P\left( N_{EF} = 1 \right), \quad N
\geq 2,
\end{equation}
which together with (\ref{infinity}) shows that $N_{EF}$ is
stochastically neither larger nor smaller than $N_{OP}.$
\end{proof}



\section*{Acknowledgements}
 We thank Cor Veenman and Gert Jacobusse for helpful discussions.

\end{document}